\documentclass{sig-alternate}
\usepackage{bm,bbm}

\newtheorem{theorem}{Theorem}
\newtheorem{corollary}[theorem]{Corollary}
\newtheorem{conjecture}[theorem]{Conjecture}
\newtheorem{proposition}[theorem]{Proposition}

\newdef{example}{Example}

\def\qbinom#1#2#3{\genfrac{[}{]}{0pt}{}{#1}{#2}_{#3}}

\def\BN{\mathbbm N}
\def\BZ{\mathbbm Z}
\def\BQ{\mathbbm Q}
\def\BR{\mathbbm R}
\def\BC{\mathbbm C}
\def\BK{\mathbbm K}
\def\BO{\mathbbm O}
\def\BW{\mathbbm W}

\def\bc{\mathbf{c}}
\def\be{\mathbf{e}}
\def\bj{\mathbf{j}}
\def\bn{\mathbf{n}}
\def\bq{\mathbf{q}}
\def\bu{\mathbf{u}}
\def\bv{\mathbf{v}}
\def\bL{\mathbf{L}}
\def\bM{\mathbf{M}}
\def\bN{\mathbf{N}}
\def\bQ{\mathbf{Q}}
\def\bal{{\bm\alpha}}
\def\bom{{\bm\omega}}

\def\w{\omega}
\def\la{\langle}
\def\ra{\rangle}
\def\longto{\longrightarrow}

\def\pt{\partial}
\def\Ann{\operatorname{Ann}}

\renewcommand{\le}{\leqslant}\renewcommand{\leq}{\leqslant}

\begin{document}
%
\conferenceinfo{ISSAC'12, July 22--25, 2012,}{Grenoble, France}

\title{Twisting q-holonomic sequences by complex roots of unity}
\numberofauthors{2}
\author{
\alignauthor
Stavros Garoufalidis\titlenote{
  Supported in part by grant DMS-0805078 of the US National Science Foundation.}\\
  \affaddr{School of Mathematics}\\
  \affaddr{Georgia Institute of Technology}\\
  \affaddr{Atlanta, GA 30332-0160, USA}\\
  \email{stavros@math.gatech.edu}
\alignauthor
Christoph Koutschan\titlenote{
  Supported by the Austrian Science Fund (FWF): P20162-N18.}\\
  \affaddr{MSR-INRIA Joint Centre}\\
  \affaddr{INRIA-Saclay}\\
  \affaddr{91893 Orsay Cedex, France}\\
  \email{koutschan@risc.jku.at}
}
\date{\today}

\maketitle
\begin{abstract}
A sequence $f_n(q)$ is $q$-holonomic if it satisfies a nontrivial
linear recurrence with coefficients polynomials in $q$ and~$q^n$. Our
main theorems state that $q$-holonomicity is preserved under twisting,
i.e., replacing $q$ by $\omega q$ where $\omega$ is a complex root of
unity, and under the substitution $q \to q^{\alpha}$ where $\alpha$ is
a rational number.  Our proofs are constructive, work in the
multivariate setting of $\partial$-finite sequences and are
implemented in the Mathematica package \texttt{HolonomicFunctions}.
Our results are illustrated by twisting natural $q$-holonomic
sequences which appear in quantum topology, namely the colored Jones
polynomial of pretzel knots and twist knots.  The recurrence of the
twisted colored Jones polynomial can be used to compute the
asymptotics of the Kashaev invariant of a knot at an arbitrary complex
root of unity.
\end{abstract}



\newpage

\section{Introduction}

A univariate sequence~$\big(f_n(q)\big)_{n\in\BN}$ is
called \emph{$q$-holonomic} if it satisfies a nontrivial linear
recurrence with coefficients that are polynomials in $q$ and~$q^n$;
the indeterminate~$q$ here is assumed to be transcendental over~$\BK$
which, for the moment, is an arbitrary but fixed field of
characteristic zero.  More precisely, $f_n(q)$ is $q$-holonomic if
there exists a nonnegative integer~$d$ and bivariate polynomials
$c_j(u,v)\in\BK[u,v]$ for $j=0,\dots,d$ with $c_d(u,v)\neq0$ such that
for all $n\in\BN$ the following recurrence is satisfied:
\begin{equation}
\label{eq.qhol}
\sum_{j=0}^d c_j(q,q^n) f_{n+j}(q)=0.
\end{equation}
The notion of $q$-holonomic sequences was introduced by Zeilberger~\cite{Z} in the
early 1990s and occurs frequently in enumerative
combinatorics~\cite{FS,St} and more recently also in quantum
topology~\cite{GL1}. Zeilberger and Wilf~\cite{WZ} proved a
\emph{Fundamental Theorem} (i.e., multisums of $q$-proper hyper\-geo\-metric
terms are $q$-holonomic), and their proof was algorithmic and
compu\-ter-implemented; an excellent introduction into the subject is
given in~\cite{PWZ}.

It is well known that the class of $q$-holonomic sequences is closed
under certain operations that include addition and
multiplication~\cite{KoepfRajkovicMarinkovic07,KauersKoutschan09}.
These operations can be executed algorithmically on the level of
recurrences, i.e., given recurrences for two $q$-holonomic sequences
$f_n(q)$ and $g_n(q)$, a recurrence for $f_n(q)+g_n(q)$ and one for
$f_n(q)\cdot g_n(q)$ can be computed; see the packages
\texttt{qGeneratingFunctions}~\cite{KauersKoutschan09} and
\texttt{HolonomicFunctions}~\cite{Ko1} for implementations in
Mathematica, as well as the Maple package
\texttt{Mgfun}~\cite{Chyzak98}.  

The aim of the present article is to establish two new closure
properties for $q$-holonomic sequences. The first one is
\emph{twisting by roots of unity}: for a given complex
number~$\w\in\BC$, we call $f_n(\w q)$ the \emph{twist} of the
sequence~$f_n(q)$ by~$\w$. Closure of $q$-holonomicity under twisting
by~$\w$ requires that $\w$ is a complex root of unity as the example
of $f_n(q)=q^{n^2}$ shows; see Remark 1.5 and Section 3.2
of~\cite{GV}. The second closure property introduced here is the
substitution of~$q$ by $q^{\alpha}$ where $\alpha$ is a rational
number.

So far the discussion was about univariate sequences. A generalization
of $q$-holonomy to a multivariate setting was given
in~\cite{Sabbah90}.  The theory of $q$-holonomic sequences parallels
to the geometric theory of holonomic systems, see~\cite{SST} and
references therein.  A different generalization of univariate
$q$-holonomic sequences to several variables is given by the class of
\emph{$\pt$-finite functions} introduced by
Chyzak~\cite{Chy,Chyzak00}. This notion is a little weaker than
$q$-holonomicity but very useful in practice, as the execution of
closure properties (e.g., addition and multiplication) is rather
simple and requires merely linear algebra. In our $q$-setting the
definition can be stated as follows: a multivariate sequence
$f_{\bn}(\bq)$ is $\pt$-finite if for every variable
$\bn=n_1,\dots,n_r$ it satisfies a linear recurrence of the
form~\eqref{eq.qhol}:
\begin{equation}
\label{eq.qholmult}
  \sum_{j=0}^{d_k} c_{k,j}(\bq,q_{a_1}^{n_1},\dots,q_{a_r}^{n_r}) f_{\bn+j\be_k}(\bq) = 0
\end{equation}
for $k=1,\dots,r$.
We use bold letters for vectors and denote by $\be_k$ the $k$-th unit
vector of length~$r$.  As above, the $d_k$'s are nonnegative integers
and the $c_{k,j}$'s are multivariate polynomials in $\BK[\bu,\bv]$
with $c_{k,d_k}\neq 0$. The indeterminates $\bq=q_1,\dots,q_s$ with
$1\le s\le r$ are assumed to be transcendental over~$\BK$ and the
indices $a_1,\dots,a_r$ need to be between~$1$ and~$s$.  In most
applications just a single indeterminate~$q$ occurs, i.e., $s=1$. From
the definitions~\eqref{eq.qhol} and~\eqref{eq.qholmult} it is
immediately clear that for univariate sequences (i.e., for $r=1$) the notions
$q$-holonomic and $\pt$-finite coincide. A more detailed exposition on
holonomy and $\pt$-finiteness can be found in~\cite{Ko1}.

The twist of the sequence $f_{\bn}(\bq)$ by complex numbers
$\bom=\w_1,\dots,\w_s$ is the sequence
$f_{\bn}(\w_1q_1,\dots,\w_sq_s)$; Theorem~\ref{thm.twist} states that
$\pt$-finiteness is preserved under twisting by complex roots of
unity. On the other hand, one may be interested in the sequence
$f_{\bn}(q_1^{\alpha_1},\dots,q_s^{\alpha_s})$ for rational numbers
$\alpha_1,\dots,\alpha_s\in\BQ$; Theorem~\ref{thm.subs} states that
$\pt$-finiteness is also preserved under this substitution.

A motivation for our work was the effective computation of the
expansion of the Kashaev invariant of a knot, i.e., its colored Jones
polynomial around complex roots of unity, that was initiated by
Zagier~\cite{DGLZ,Za2}; see also~\cite{Ga4}.  Using our results, such
an expansion can now be achieved and will be the focus of several
separate publications~\cite{DG,GZ}. More details and some examples are
given in Section~\ref{sec.QT}.

\section{Twisting Preserves {\large $\pt$}-Finiteness}

\subsection{Operator Notation and Left Ideals}

To state our results, it will be helpful to write recurrences
like~\eqref{eq.qhol} in operator form. For this purpose consider the
operators~$L$ and $M$ which act on a sequence $f_n(q)$ by
\begin{align*}
  L f_n(q) & = f_{n+1}(q),\\
  M f_n(q) & = q^n f_n(q),
\end{align*}
and satisfy the $q$-commutation relation $LM=qML$. The noncommutative
algebra that is generated by $L$ and $M$ modulo $q$-commutation is
denoted by $\BW=\BK(q)[M]\la L\ra$ and is called the \emph{first
  $q$-Weyl algebra}. If one wants to allow division by~$M$ then it is
convenient to utilize a noncommutative \emph{Ore algebra}
(see~\cite{Chy,Chyzak00} for more details) which is denoted by
$\BO=\BK(q,M)\langle L\rangle$. Clearly the inclusion $\BW\subset\BO$
holds.

Similarly, for representing the system of
recurrences~\eqref{eq.qholmult}, the operators $\bL=L_1,\dots,L_r$ and
$\bM=M_1,\dots,M_r$ are introduced, which act on a multivariate
sequence $f_{\bn}(\bq)$ by
\begin{equation}\label{eq.LMq}
\begin{split}
  L_kf_{\bn}(\bq) & = f_{\bn+\be_k}(\bq),\\
  M_kf_{\bn}(\bq) & = q_{a_k}^{n_k} f_{\bn}(\bq),
\end{split}
\end{equation}
for $k=1,\dots,r$ and with the same notation as
in~\eqref{eq.qholmult}. Again the above operators $q$-commute, i.e.,
they satisfy
\begin{align*}
  L_kM_k & = q_{a_k}M_kL_k,\\
  L_jM_k & = M_kL_j \quad\text{for } j\neq k.
\end{align*} 
More generally, we can state the $q$-commutation for arbitrary
expressions in~$\bM$:
\[
  L_kF(\bM)=F(M_1,\dots,M_{k-1},q_{a_k}M_k,M_{k+1},\dots,M_r)L_k.
\]
In operator form, Equation~\eqref{eq.qholmult} is written as
$P_kf=0$ where 
\begin{equation}
\label{eq.Pkf}
P_k=\sum_{j=0}^{d_k} c_{k,j}(\bq,\bM)L_k^j 
\end{equation}
for $k=1,\dots,r$. The operators $P_1,\dots,P_r$ are regarded as
elements of the Ore algebra $\BO=\BK(\bq,\bM)\la\bL\ra$. The
algebra~$\BO$ can be viewed as the multivariate polynomial ring in the
indeterminates $L_1,\dots,L_r$ with coefficient field being the
rational functions in $\bq$ and $\bM$, subject to the above stated
$q$-commutation relations.  Given a multivariate sequence
$f_{\bn}(\bq)$, the set
\[
  \Ann_{\BO}(f)=\{P \in \BO \, \mid P f=0 \}
\]
is a left ideal of $\BO$, the so-called annihilator of~$f$ with
respect to the algebra~$\BO$. Left ideals in~$\BO$ have well-defined 
{\em dimension} and {\em rank} which can be computed for instance by
(left) Gr\"obner bases. In this terminology, a
multivariate sequence $f_{\bn}(\bq)$ is $\pt$-{\em finite} with
respect to~$\BO$ if $\Ann_{\BO}(f)$ is a zero-dimensional left ideal in~$\BO$.
For example, if $f_{\bn}(\bq)$ satisfies~\eqref{eq.qholmult}, then it
is annihilated by the operators $P_1,\dots, P_r$ of
Equation~\eqref{eq.Pkf}.  The latter generate a zero-dimensional ideal
of rank at most~$\prod_{k=1}^rd_k$.  Note, however, that the set
$\{P_1,\dots,P_r\}$ is not a left Gr\"obner basis of that ideal in
general (Buchberger's product criterion does not hold in
noncommutative rings).

\subsection{Main Theorems}

We have now prepared the stage for stating our main results.  To keep
the presentation concise, it is assumed from now on that the
field~$\BK$ contains all complex roots of unity.

\begin{theorem}
\label{thm.twist}
Let $f_{\bn}(\bq)=f_{n_1,\dots,n_r}(q_1,\dots,q_s)$ be a multivariate
$\pt$-finite sequence, and let $\w_j\in\BC$ be an $m_j$-th root of unity 
for $j=1,\dots,s$. Then, the twisted sequence $g_{\bn}(\bq)=
f_{\bn}(\w_1q_1,\dots,\w_sq_s)$ is $\pt$-finite as well. 

Moreover, let $I$ be a zero-dimensional left ideal of rank~$R$
such that $If=0$. From a generating set of~$I$, a Gr\"obner basis 
of a zero-dimensional left ideal~$J$ with $Jg=0$ can be obtained
and its rank is at most $R\cdot m_{a_1}\cdots m_{a_r}$.
\end{theorem}
\begin{proof}
With the notation introduced in~\eqref{eq.qholmult} and~\eqref{eq.LMq}
we fix the Ore algebra $\BO=\BK(\bq,\bM)\langle\bL\rangle$ so that $I$
is a left ideal in~$\BO$. We now shall show that sufficiently many
operators in~$\BO$ can be found which annihilate the
sequence~$g_{\bn}(\bq)$. A naive attempt to obtain some recurrences
for~$g$ is to substitute $q_j$ by $\w_jq_j$ (for $1\leq j\leq s$) in
the recurrences for~$f$.  Indeed, the result are valid recurrences
for~$g$, but in general they cannot be represented in the
algebra~$\BO$ since they contain terms of the form $\w_j^{n_k}$.
However, for an operator $P\in I$ this substitution is admissible (in
the sense that the result is in~$\BO$) if for each~$k$ the variable
$M_k$ appears in~$P$ only with powers that are multiples of~$m_{a_k}$
(for sake of readability we will write $m(k)$ instead of $m_{a_k}$).
The idea of the proof is to show that such operators exist and that
they generate a zero-dimensional ideal of rank at most $R\cdot
m(1)\cdots m(r)=:\tilde{R}$.

First we introduce a new set of variables $\bN=N_1,\dots,N_r$ such
that $N_k=M_k^{m(k)}$. In this notation the goal is to obtain a set of
generators for the left ideal
\[
  J=I \cap \BK(\bq,\bN)\langle\bL\rangle.
\]
For this purpose, fix~$k$ and consider an ansatz operator of the form
\[
  A = \sum_{j=0}^d c_j(\bq,\bN)L_k^j
\]
where the unknowns $\bc=c_0,\dots,c_d$ are assumed to be rational
functions in $\bq$ and~$\bN$. The remainder of $A$ modulo the left
ideal~$I$ can be computed by reducing it with a left Gr\"obner basis
of~$I$. After clearing denominators, this remainder is a linear
combination of $R$ different power products $\bL^\bal$; its
coefficients are polynomials in $\bq$ and $\bM$, and in the unknowns
$\bc$ which occur linearly. The claim that $A$ be an annihilating
operator for~$f$ is achieved by equating all those coefficients to
zero. This yields a system of $R$ equations in the unknowns~$\bc$.  By
making use of the new variables $\bN$ and simple rewriting, it can be
achieved that the degree of $M_k$ is smaller than $m(k)$ for $1\leq
k\leq r$.  Coefficient comparison w.r.t. the variables $\bM$ enforces
that the unknowns $\bc$ depend only on $\bq$ and $\bN$, and converts
each equation into a set of at most $m(1)\cdots m(r)$
equations. Choosing $d=\tilde{R}$ in $A$
therefore produces a linear system with $d$ equations in $d+1$
unknowns. Thus the existence of a nontrivial solution is guaranteed.
The substitutions $q_j\to\w_jq_j$ can now be performed without
problems and yield an annihilating operator for~$g$. Repeating the
above procedure for $k=1,\dots,r$ shows that $g$ is $\pt$-finite.

However, in practice one would not proceed along these lines. Instead
of pure recurrence operators~$A$ (i.e., univariate polynomials
in~$\BO$), it is advantageous to loop over the support of~$A$ and
increase it according to the FGLM algorithm (this is made explicit in
Algorithm~1 below). This procedure guarantees that the resulting
operators form a Gr\"obner basis, and at the same time shows that the
rank of the ideal they generate is at most $\tilde{R}$.  For the
contrary, let $R'$ denote the rank of~$J$ and assume that it is
strictly greater than~$\tilde{R}$; this means that a Gr\"obner basis
of $J$ has $R'$ irreducible monomials under its stairs, i.e., there is
no operator in $J$ whose support is a subset of these monomials.  On
the other hand, an ansatz~$A$ (as above) whose support consists of all
irreducible monomials will lead to a linear system with $\tilde{R}$
equations and $R'$ unknowns. By the assumption $R'>\tilde{R}$ a
nontrivial solution exists, in contradiction to the fact that the
support of $A$ consists of irreducible monomials only.
\end{proof}

Since many applications deal with sequences in a single variable,
and in order to justify the title of this paper, the following
corollary is stated explicitly, even though it is a trivial
consequence of Theorem~\ref{thm.twist}.
\begin{corollary}\label{cor.twist}
Let $f_n(q)$ be a $q$-holonomic sequence that satisfies a recurrence
of the form~\eqref{eq.qhol} of order~$d$. Then for any root of unity
$\omega\in\BC$ of order~$m$ the sequence~$f_n(\omega q)$ is
$q$-holonomic as well and satisfies a recurrence of order at most
$m\cdot d$.
\end{corollary}

In \cite[Thm. 1.5]{GV} it was shown that the specialization of a
$q$-holonomic sequence $f_n(q) \in \BZ[q^{\pm 1}]$ to a complex root
of unity~$\w$ is a holonomic sequence, in other words, that $f_n(\w)$
satisfies a linear recurrence with coefficients polynomials
in~$n$. The present paper reduces the proof of the above result to the
case of~$\w=1$.

\begin{theorem}
\label{thm.subs}
Let $f_{\bn}(\bq)=f_{n_1,\dots,n_r}(q_1,\dots,q_s)$ be a multivariate
$\pt$-finite sequence, and let $\alpha_1,\dots,\alpha_s\in\BQ$. Then, the
sequence $g_{\bn}(\bq)= f_{\bn}(q_1^{\alpha_1},\dots,q_s^{\alpha_s})$ is
$\pt$-finite as well.

Moreover, let $I$ be a zero-dimensional left ideal of rank~$R$ such
that $If=0$. From a generating set of~$I$, a Gr\"obner basis of a
zero-dimensional left ideal~$J$ with $Jg=0$ can be obtained and its
rank is at most $R\cdot m_1\cdots m_s\cdot m_{a_1}\cdots m_{a_r}$,
where $m_j\in\BN$ denotes the denominator of $\alpha_j$.
\end{theorem}
\begin{proof}
Employing the notation from~\eqref{eq.qholmult} and~\eqref{eq.LMq} so
that $I$ is a left ideal in $\BO=\BK(\bq,\bM)\langle\bL\rangle$, it
has to be shown that there are sufficiently many elements in~$I$ for
which the result of the substitutions $q_j\to q_j^{\alpha_j}$, $1\leq
j\leq s$, is still in~$\BO$. This condition is equivalent to claiming
that all powers of $q_j$ are divisible by $m_j$ and that all powers of
$M_k$ are multiples of~$m_{a_k}$, for $1\leq j\leq s$ and $1\leq k\leq r$.

The rest of the proof is analogous to the proof of
Theorem~\ref{thm.twist}.  The only difference is that in addition
to~$\bN$, one has to introduce a second set of new variables
$\bQ=Q_1,\dots,Q_s$ such that $Q_j=q_j^{m_j}$, and that the
coefficient comparison then has to be performed w.r.t. $\bM$
and~$\bq$.
\end{proof}

\begin{corollary}\label{cor.subs}
Let $f_n(q)$ be a $q$-holonomic sequence that satisfies a recurrence
of the form~\eqref{eq.qhol} of order~$d$. Then for $\alpha\in\BQ$ the
sequence~$f_n(q^\alpha)$ is $q$-holonomic as well and satisfies a
recurrence of order at most $m^2\cdot d$, where $m\in\BN$ is the
denominator of~$\alpha$.
\end{corollary}

It is now natural to ask whether Corollaries~\ref{cor.twist}
and~\ref{cor.subs} can be extended to $q$-holonomic sequences in more
than one variable.  Unfortunately the study of multivariate
$q$-holonomic sequences is much more involved (we even didn't give a
precise definition in this paper), and therefore the following
statement appears without proof; it is a stronger version of
Theorems~\ref{thm.twist} and~\ref{thm.subs}.
\begin{conjecture}
Multivariate $q$-holonomic sequences are closed under twisting by
complex roots of unity and under substitutions of the form $q\to
q^{\alpha}$ for $\alpha\in\BQ$.
\end{conjecture}

At this point it may be beneficial to discuss some simple examples to
illustrate Theorems~\ref{thm.twist} and~\ref{thm.subs} and their
implementation in our software package.  Recall the definitions for
the $q$-Pochhammer symbol
\[
  (a;q)_n := \prod_{k=0}^{n-1} \big(1-aq^k\big)
\]
and the $q$-binomial coefficient
\[
  \qbinom{n}{k}{q} := \frac{(q;q)_n}{(q;q)_k(q;q)_{n-k}}.
\]

\begin{example}
\label{ex.cqbc}
Let $f_n(q)$ be the central $q$-binomial coefficient~$\qbinom{2n}{n}{q}$.
It satisfies the recurrence
\[
  (1-q^{n+1})f_{n+1}(q) = (1+q^{n+1}-q^{2n+1}-q^{3n+2})f_n(q)
\]
which translates to the operator
\begin{equation}\label{eq.op.cqbc}
  (qM-1)L-q^2M^3-qM^2+qM+1.
\end{equation}
We choose $\w=-1$; the substitution $q\to -q$ in the above operator
is not admissible because of the odd powers of~$M$. On the other hand,
Theorem~\ref{thm.twist} guarantees that $f_n(-q)$ is also $q$-holonomic.
Indeed, the twisted sequence~$f_n(-q)$ is annihilated by the operator
\begin{multline*}
  \left(q^4 M^2-1\right)L^2 + 
  \left(\left(q^7-q^6\right)M^4-q+1\right)L - {}\\
  q^7M^6-\left(q^6-q^5+q^4\right)M^4+\left(q^4-q^3+q^2\right)M^2+q.
\end{multline*}
Note that it contains only even powers of~$M$, at the cost of
increasing the order.  Using the Mathematica
package \texttt{HolonomicFunctions}, these results can be obtained
by the following commands:
\begin{verbatim}
qbin = Annihilator[QBinomial[2n, n, q], QS[qn,q^n]]
DFiniteQSubstitute[qbin, {q, 2}]
\end{verbatim}
The first line determines the input operator~\eqref{eq.op.cqbc} from
the given mathematical expression. The second line computes the
twisted recurrence; the substitution is given as a pair $(q,m)$ and by
default $\w=e^{2\pi i/m}$ is chosen.
\end{example}

\begin{example}
\label{ex.qp}
The $q$-Pochhammer symbol satisfies the simple recurrence
\[
  (q;q)_{n+1} = (1-q^{n+1})(q;q)_n.
\]
We want to study the twisted sequence $(\w q;\w q)_n$ for $\w$
being a third root of unity. Therefore we have to compute a recurrence
for $(q;q)_n$ in which all exponents of~$M=q^n$ are divisible by~$3$:
\begin{equation}
\label{eq.rec.twqp}
\begin{split}
  & (q;q)_{n+3} - \left(q^2+q+1\right)(q;q)_{n+2} + {}\\
  & \qquad \left(q^3+q^2+q\right)(q;q)_{n+1} + \left(q^{3n+6}-q^3\right)(q;q)_n = 0.
\end{split}
\end{equation}
Substituting $q\to \w q$ into~\eqref{eq.rec.twqp} delivers a
$q$-holonomic recurrence for the twist $(\w q;\w q)_n$. The commands
to compute it are the following:
\begin{verbatim}
qp = Annihilator[QPochhammer[q, q, n], QS[qn,q^n]]
DFiniteQSubstitute[qp, {q, 3},
   Return -> Backsubstitution]
\end{verbatim}
The option \texttt{Return -> Backsubstitution} in this instance tells
the program to return the recurrence before performing the
substitution $q\to e^{2\pi i/3}q$ (see the last but one line of
Algorithm~1); this is exactly recurrence~\eqref{eq.rec.twqp} in
operator form.
\end{example}

\begin{example}
\label{ex.qp2}
The substitution $q\to\sqrt{q}$ is performed on the $q$-Pochhammer
symbol $(q;q)_n$ (see Example~\ref{ex.qp}). Theorem~\ref{thm.subs}
predicts that the resulting recurrence is of order at most~$4$,
which is sharp in this case. As an intermediate result, the operator
\begin{multline*}
  L^4 - (q^2+1)L^3 - (q^8M^2+q^6M^2-q^4-q^2)L \\
  - q^{10}M^4+q^8M^2+q^6M^2-q^4
\end{multline*}
is found in the annihilator of $(q;q)_n$. Note that both $q$ and $M$
appear with even powers. The final result is the recurrence
\begin{multline*}
  f_{n+4} - (q+1) f_{n+3} - (q^{n+4}+q^{n+3}-q^2-q) f_{n+1} \\
  + \left(-q^{2 n+5}+q^{n+4}+q^{n+3}-q^2\right)f_n=0
\end{multline*}
where $f_n=\left(\sqrt{q};\sqrt{q}\right)_n$.
This recurrence is obtained as the output of the command
\begin{verbatim}
DFiniteQSubstitute[qp, {q, 1, 2}]
\end{verbatim}
where the triple $(q,m,k)$ encodes the substitution
$q\to\omega q^{1/k}$ with $\omega=e^{2\pi i/m}$.
\end{example}

\subsection{Algorithms}

The proof of Theorem~\ref{thm.twist} gives an algorithm to construct
the left ideal~$J$ of annihilating operators for the twisted sequence.
To formulate this algorithm in pseudo-code, the notations
from~\eqref{eq.qholmult} and~\eqref{eq.LMq} and from
Theorem~\ref{thm.twist} are employed; additionally, if $T$ is a set, we
refer to its elements by $\{T_1,T_2,\dots\}$, and we use
$\mathrm{lm}_{\prec}(P)$ to denote the leading monomial of the
operator~$P$ with respect to the monomial order~$\prec$.

\medskip
\textsc{Algorithm 1.}
\vskip 1mm\hrule\vskip 1mm
\noindent\textbf{Input:} 
\hfill\parbox[t]{70mm}{
$r,s\in\BN$,\\
for $1\leq j\leq s$: $m_j\in\BN$, $\w_j\in\BC$ with $\w_j^{m_j}=1$ and\\
$\w_j^\ell\neq1$ for all $\ell<m_j$,\\
$\BO=\BK(q_1,\dots,q_s,M_1,\dots,M_r)\langle L_1,\dots,L_r\rangle$,\\
a monomial order $\prec$ for $\BO$,\\
a finite set $F\subset\BO$ such that $F$ is a left Gr\"obner basis w.r.t.~$\prec$
and the left ideal ${}_{\BO}\langle F\rangle$ is zero-dimensional}\\[1mm]
\textbf{Output:}
\hfill\parbox[t]{70mm}{
a finite set $G\subset\BO$ such that $G$ is a left Gr\"obner basis w.r.t.~$\prec$
and such that for any sequence $f_{\bn}(q_1,\dots,q_s)$ with $F(f_{\bn}(\bq))=0$ we have
$G(f_{\bn}(\w_1q_1,\dots,\w_sq_s))=0$}
\vskip 1mm
\hrule
\vskip 1mm
\noindent $G=\emptyset$\\
$U=\text{set of monomials under the stairs of }F$\\
$T=\{1\}$\\
$V=\emptyset$\\
\textbf{while} $T\neq\emptyset$\\
\phantom{M} $T_0=\min_{\prec}T$\\
\phantom{M} $T=T\setminus\{T_0\}$\\
\phantom{M} $A=c_0T_0+\sum_{j=1}^{|V|} c_jV_j$\\
\phantom{M} $A'=A\text{ reduced with }F$\\
\phantom{M} clear denominators of $A'$\\
\phantom{M} substitute $M_k^a\to M_k^{a\!\!\mod m(k)}N_k^{\lfloor a/m(k)\rfloor}$ in $A'$\\
\phantom{M} write $A'$ as $\sum_{i=1}^{|U|}\sum_{j_1=0}^{m(1)-1}\cdots\sum_{j_r=0}^{m(r)-1} d_{i,\bj}M_1^{j_1}\cdots M_r^{j_r}U_i$\\
\phantom{M} equate all $d_{i,\bj}$ to zero\\
\phantom{M} solve this linear system for $c_0,\dots,c_{|V|}$ over $\BK(\bq,\bN)$\\
\phantom{M} \textbf{if} a solution exists \textbf{then}\\
\phantom{MM} substitute the solution into $A$\\
\phantom{MM} $G=G\cup\{A\}$\\
\phantom{MM} $T=T\cup\{T_0L_k: 1\leq k\leq r\}$\\
\phantom{MM} $T=T\setminus\{T_j: 1\leq j\leq |T| \land \exists_k\, \mathrm{lm}_{\prec}(G_k)\mid T_j\}$\\
\phantom{M} \textbf{else}\\
\phantom{MM} $V=V\cup\{T_0\}$\\
substitute $N_k\to M_k^{m(k)}$ and $q_j\to \w_jq_j$ in $G$\\
\textbf{return} $G$
\medskip

Similarly, Theorem~\ref{thm.subs} yields Algorithm~2 which, however,
is just a light variation of Algorithm~1 and therefore not displayed
explicitly here. Both algorithms are implemented in
\texttt{HolonomicFunctions} as the command
\texttt{DFiniteQSubstitute}, see~\cite{Ko2} and
Examples~\ref{ex.cqbc}--\ref{ex.qp2}.

With slight modifications Algorithms~1 and~2 can be applied to
inhomogeneous recurrences as well. Algebraically, an inhomogeneous
recurrence of the form
\[
  \sum_{j=0}^d c_j(q,q^n) f_{n+j}(q)=b(q,q^n)
\]
can be represented as $\big(\sum_{j=0}^d c_jL^j,b\big)$ in the left
module~$\BO^2$, modulo the relation~$(0,L-1)$. To make the algorithms
work a POT ordering has to be used. The option \texttt{ModuleBasis} of
the command \texttt{DFiniteQSubstitute} serves this purpose.

Given a root of unity~$\w\in\BC$ and a univariate operator $P\in\BW$
such that $P(f_n(q))=0$ for some sequence~$f_n(q)$, let
$\tau_{\w}(P)\in\BW$ denote the annihilating operator for the twisted
sequence $f_n(\w q)$ that is produced by Algorithm~1 (in order to
represent its output in~$\BW$, one has to clear denominators). 
Additionally we claim that $\tau_{\w}(P)=\sum_{j=0}^d c_j(q,M)L^j$ is
\emph{content-free}, i.e., $\gcd(c_0,\dots,c_d)=1$. The following
result about the nature of $\tau_w(P)$ is easily obtained.
\begin{proposition}
\label{prop.factors}
Let 
\[
  P(M,L,q) = \sum_{j=0}^d c_j(q,M)L^j \in\BW
\]
such that $\gcd(c_0,\dots,c_d)=1$ and let $\w\in\BC$ be a root of
unity of order~$m$.  Define $\ell\in\BN$ to be the largest integer such
that $P\in\BK(q)[M^\ell]\langle L\rangle$.  Then
\[
  Q(M,L)\big(\tau_{\w}(P)\big)(M,L,\w^{-1}) = R(M)\! \prod_{k=1}^{m/\!\gcd(\ell,m)}\!\!\!\!\! P(\w^kM,L,1)
\]
for some polynomial $Q\in\BK[M,L]$ and some rational function
$R\in\BK(M)$.
\end{proposition}

\subsection{Behavior of the Newton Polygon Under\\ Twisting}

In this section it is studied how the Newton polygon of a univariate
operator behaves under twisting.  Following~\cite{Ga7}, consider the
\emph{Newton polygon}~$N(P)$ of an operator $P\in\BW$, i.e., the
convex hull of the exponents $(a,b)$ of the monomials $M^bL^a$
of~$P$. The Newton polygon of a (possibly inhomogeneous) recurrence
$P(f_n(q))=b(q,q^n)$, $P\in\BW$, is defined to be~$N(P)$.  Let $LN(P)$
denote the \emph{lower convex hull} of~$N(P)$. $LN(P)$ consists of a
finite union of non-vertical line segments together with two vertical
rays. Each line segment has a \emph{slope} and we denote by $S(P)$ the
\emph{set of slopes} of~$LN(P)$.
An example will clarify these notions.
\begin{example}
\label{ex.newton}
Consider the inhomogeneous recurrence
\begin{equation}\label{eq.reci41}
\begin{split}
& q^{2n+2} \big(q^{n+2}-1\big) \big(q^{2n+1}-1\big) f(n+2) -{}\\
& \big(q^{4n+4}-q^{3n+3}-q^{2n+3}-q^{2n+1}-q^{n+1}+1\big)\\
& \quad\times\big(q^{n+1}-1\big)^2 \big(q^{n+1}+1\big) f(n+1) +{}\\
& q^{2n+2} \big(q^n-1\big) \big(q^{2 n+3}-1\big) f(n) ={}\\
& q^{n+1} \big(q^{n+1}+1\big) \big(q^{2 n+1}-1\big) \big(q^{2 n+3}-1\big)
\end{split}
\end{equation}
whose left-hand side is $P(f_n(q))$ where the operator~$P$ is given by
\begin{align*}
  & \left(q^5M^5-q^3M^4-q^4M^3+q^2M^2\right)L^2 +{}\\
  & \left(-q^7M^7 + 2q^6M^6 + (q^6+q^4)M^5 - (q^5+q^4+q^3)M^4 -{}\right.\\
  & \quad \left.(q^4+q^3+q^2)M^3 + (q^3+q)M^2 + 2qM-1 \right)L +{}\\
  & q^5M^5-q^5M^4-q^2M^3+q^2M^2.
\end{align*}
Then $N(P)$ is the hexagon with vertex set
\[
  \big\{(0,2),(1,0),(2,2),(2,5),(1,7),(0,5)\big\},
\]
which corresponds to the smallest polygon depicted in
Figure~\ref{fig.newtonplot}.  The lower Newton polygon $LN(P)$
consists of the two line segments which connect the points $(0,2)$,
$(1,0)$, and~$(2,2)$, as well as the two vertical rays starting from
$(0,2)$ and~$(2,2)$. The set of slopes $S(P)$ is easily seen to be
$\{-2,2\}$. 
\end{example}

\begin{proposition}
\label{prop.slopes}
Fix $P \in \BW$ and $\w\in\BC$ a complex $m$-th root of unity. Then
$\tau_{\w}(P) \in \BK(q)[M^m]\la L\ra$ and $S(P)\subset
S(\tau_{\w}(P))$.
\end{proposition}
\begin{proof}
By definition, our algorithm finds a polynomial $Q \in \BW$ such that
$\tau_{\w}(P)=QP \in \BK(q)[M^m]\la L\ra$.  In \cite[Prop.2.2]{Ga7} it
is shown that $LN(QP)=LN(Q)+LN(P)$, where the plus operation is the
\emph{Minkowski sum}. Since the slopes of the Minkowski sum is the
union of the slopes, it follows that $S(P) \subset S(\tau_{\w}(P))$.
\end{proof}
Using Proposition~\ref{prop.factors} one even gets equality instead of
the inclusion. However, if the Newton polygons of inhomogeneous
recurrences are considered, the set of slopes can strictly grow under
twisting; this will be demonstrated in Section~\ref{sub.41}.

Note that every edge of $N(P)$ is either an edge of $LN(P)$, or an edge
of $U\!N(P)$ (the \emph{upper convex hull} of the exponents of~$P$), or a 
vertical edge. Proposition~\ref{prop.slopes} applies to $U\!N(P)$ as well,
by reversing $q$ to~$1/q$.

\section{Applications in Quantum Topo-\\ logy}
\label{sec.QT}

\subsection{The Colored Jones Polynomial of a Knot}
 
Quantum knot theory is a natural source of $q$-holonomic sequences. 
A knot~$K$ is the smooth embedding of a circle in 3-dimensional space~$\BR^3$, 
up to isotopy. The {\em colored Jones polynomial} 
\[
  \big(J_{K,n}(q)\big)_{n\in\BN} \in \big(\BZ[q^{\pm 1}]\big)^{\BN}
\]
of a knot~$K$ is a sequence of Laurent 
polynomials with the normalization that
$J_{K,1}(q)=1$ and $J_{\text{Unknot},n}(q)=1$ for all~$n$. $J_{K,2}(q)$ is the 
famous \emph{Jones polynomial}~\cite{Jo}.  
For an introduction to the polynomial invariants of knots that
originate in quantum topology see~\cite{Kf,Jo,Tu1,Tu2} and the book~\cite{Ja}
where all the details of the quantum group theory can be found. Up-to-date 
computations of several polynomial invariants of knots
are available in~\cite{B-N}. The colored Jones polynomial $J_{K,n}(q)$ is 
a $q$-holonomic sequence~\cite{GL1};
as a canonical (homogeneous) recurrence relation
we choose the one with minimal order; this is the so-called 
\emph{noncommutative $A$-polynomial} $A_K(M,L,q) \in \BW$ of a knot~$K$~\cite{Ga1}. 
An inhomogeneous recurrence is often available, typically of 
smaller size~\cite{GS2,Ga3}. Theorem~\ref{thm.twist} has the following corollary.
\begin{corollary}
There exists a twisting map
\[
  \text{Knots} \, \times \{\text{complex roots of $1$}\} \longto \BW
\]
defined by $(K,\w) \mapsto A_{K,\w}(M,L,q)$
with the following properties:
\begin{itemize}
\item[(a)] $A_{K,\w}(M,L,q)=\tau_{\w}(A_{K,1}(M,L,q))$ and the base case
$A_{K,1}(M,L,q)=A_K(M,L,q)$ is determined by the colored Jones polynomial
$J_{K,n}(q)$.
\item[(b)] For every complex root of unity $\w$, $J_{K,n}(\w q)$ is annihilated
by $A_{K,\w}(M,L,q)$.
\item[(c)] If $\w$ has order $m$, then $ A_{K,\w}(M,L,q) \in \BK(q)[M^m]\la L\ra$.
\end{itemize}
\end{corollary}
The above corollary also holds for the inhomogeneous non-commutative 
$A$-polynomial, too.

\subsection{Examples of Noncommutative {\large $A$}-Polyno\-mi\-als of Knots}

Although the noncommutative $A$-polynomial of a knot is essentially a 
three-variate polynomial, it is a difficult one to compute or to guess.
In fact, a conjectured two-variate specialization of it, the so-called
\emph{$A$-polynomial} of a knot (defined in \cite{CCGLS}) is already hard to 
compute and even unknown for some knots with only $9$~crossings. For an updated list of 
$A$-polynomials of knots, see~\cite{Cu}. There are two 1-parameter families
of knots with known $A$-polynomials, namely the twist knots $K_p$~\cite{HS}
and the $(-2,3,3+2p)$ pretzel knots $K\!P_p$~\cite{GM}, 
depicted in Figure~\ref{fig.knots}.
\begin{figure}
\centering
\epsfig{file=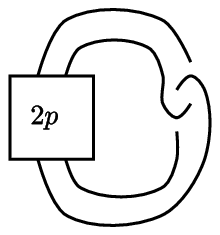, width=.95in}
\qquad
\epsfig{file=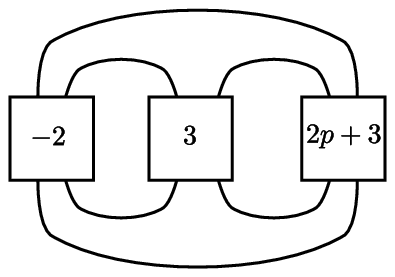, width=1.5in}\\[1em]
\epsfig{file=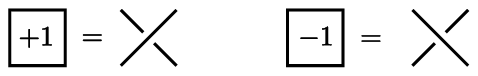, width=1.5in}
\caption{
Twist knot~$K_p$ (left) and $(-2,3,2p+3)$ pretzel knot $K\!P_p$ (right)
where an integer~$m$ inside a box indicates the number 
of $|m|$ half-twists, right-handed (if $m>0$) or left-handed (if $m<0$).}
\label{fig.knots}
\end{figure}
For these two families of knots, the (inhomogeneous) noncommutative
$A$-polynomials have been computed or guessed only for a few
particular values of the parameter~$p$.  For the twist knots $K_p$,
they were computed with a certificate in~\cite{GS2} for
$p=-14,\dots,15$.  For the pretzel knots $K\!P_p=(-2,3,3+3p)$, they
were guessed by the authors in~\cite{GK} for $p=-5,\dots,5$.  The
results of twisting these recurrences by $\w=-1$ can be found in
\begin{center}
\texttt{www.math.gatech.edu/$\sim$stavros/publications/} \\
\texttt{twisting.qholonomic.data/}
\end{center}

\subsection{The {\large $4_1$} Knot}
\label{sub.41}

As a case study we investigate the twist knot~$K_{-1}$ which appears
as knot~$4_1$ in the knot atlas~\cite{B-N}. The inhomogeneous
recurrence for its colored Jones polynomial is given
by~\eqref{eq.reci41}; see \cite{GL1,Ga1}. 
Table~\ref{tab.41data} shows the sizes and
exponents of the twisted recurrences and demonstrates that they grow
rapidly with the order~$m$ of the root of unity.
\begin{table}
\centering
\caption{Data for the twisted inhomogeneous recurrences of the $4_1$ knot;
the integer~$m$ denotes the order of the root of unity by which the recurrence
is twisted and its size is given in terms of Mathematica ByteCount.}
\label{tab.41data}
\begin{tabular}{|c|c|c|c|c|c|} \hline
$m$ & $1$ & $2$ & $3$ & $4$ & $5$ \\ \hline
size in KB & $3$ & $80$ & $3867$ & $13460$ & $68477$ \\ \hline
$q$-exponent & $7$ & $58$ & $327$ & $698$ & $1661$ \\ \hline
$L$-exponent & $2$ & $5$ & $8$ & $11$ & $14$ \\ \hline
$M$-exponent & $7$ & $22$ & $81$ & $124$ & $235$ \\ \hline
\end{tabular}
\end{table}

The Newton polygons of the twisted (inhomogeneous) recurrences for the
orders $m=1,\dots,5$ are given in Figure~\ref{fig.newtonplot} (recall
that for the Newton polygon of an inhomogeneous recurrence, we
consider just the homogeneous part of that recurrence). They are
plotted in $(L,M^m)$ coordinates, which means that a point $(a,b)$ in
the Newton polygon for a certain~$m$ represents the monomial
$M^{bm}L^a$. Note that the set of slopes is $\{-2,2\}$ for the input
recurrence~\eqref{eq.reci41}, but that it is $\{-2,0,2\}$ for the
Newton polygons of the twisted recurrences.
\begin{figure}
\centering
\epsfig{file=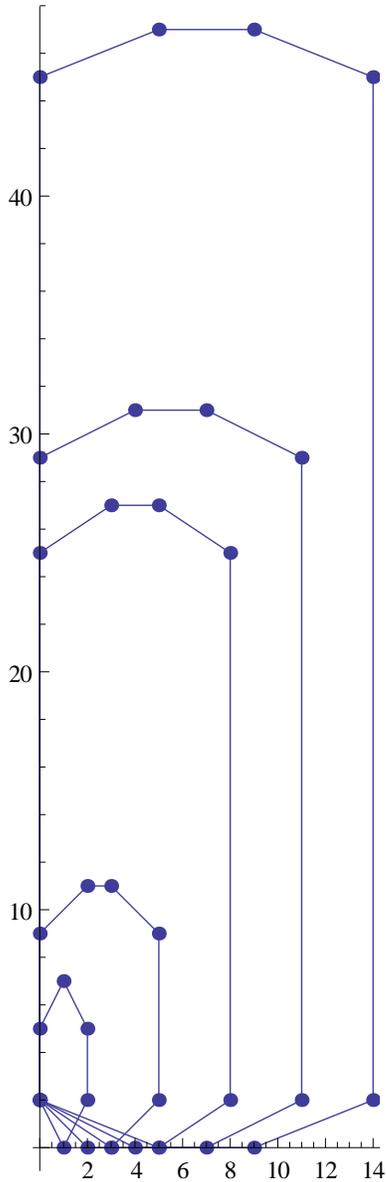, width=2in}
\caption{
The Newton polygon of the twisted (inhomogeneous) recurrences for the knot~$4_1$ in $(L,M^m)$-space;
note that the slopes appear jolted due to the use of $(L,M^m)$ coordinates.}
\label{fig.newtonplot}
\end{figure}

\subsection{An Application of Twisting in Quantum\\ Topology}

In this section we discuss in brief an application of twisting to
asymptotics questions in quantum topology. For further details and the
role of recurrences, see~\cite{DGLZ,DG,Ga8}.

The Kashaev invariant
$\la K \ra_n$ of a knot $K$ is given by \cite{Ka,MM}
$$
\la K \ra_n = J_{K,n}(e^{2 \pi i/n}).
$$
The \emph{Volume Conjecture} relates the leading asymptotics of the Kashaev invariant
to hyperbolic invariants of the knot complement. More precisely, the Volume
Conjecture states that for a hyperbolic knot $K$ we have:
$$
\lim_n \frac{1}{n} \log|\la K \ra_n|=\frac{\text{vol(K)}}{2\pi}
$$
where $\text{vol}(K)$ is the hyperbolic volume of $K$ \cite{Th}. 
It was observed by
Zagier and the first author that one can numerically compute $\la K \ra_n$
in $O(n)$ time given a recurrence relation for $J_{K,n}(q)$. Zagier raised
questions concerning the expansion of the Kashaev invariant around other roots
of unity (the original Volume Conjecture is centered around $\w=1$).
Given a recurrence relation for $J_{K,n}(\w q)$, one can compute those 
asymptotics in linear time. This will be studied in detail in forthcoming
work \cite{DG,GZ}.

\section{Acknowledgments}
The first named author wishes to thank T. Dimofte and D. Zagier for
many stimulating conversations, and the Max Planck Institute in Bonn
for their superb hospitality.  The second named author was employed by
the Research Institute for Symbolic Computation (RISC) of the Johannes
Kepler University in Linz, Austria, while carrying out the research
for the present paper.


\end{document}